\documentclass[11pt]{article}
 \usepackage{geometry} 
\usepackage{amsmath} \numberwithin{equation}{section}
\usepackage{amsthm}
\usepackage{amsfonts}
\usepackage{amssymb}
\usepackage{graphicx}
\usepackage{url}

\newtheorem{theorem}{Theorem}[section]

\theoremstyle{definition}

\newtheorem{corollary}[theorem]{Corollary}

\theoremstyle{remark}
\newtheorem{remark}[theorem]{Remark}

\numberwithin{equation}{section}

\newcommand\C{\mathbb C}         
\newcommand\R{\mathbb R}       
         
\newcommand\N{\mathbb N}         
\newcommand\Ha{\mathbb H}       
\newcommand\D{\mathbb D}     

\renewcommand\Im{\text{Im}\,}        
\renewcommand\Re{\text{Re}\,}

\begin{document}
\setlength{\parindent}{0pt}

\title{Two value ranges for symmetric self-mappings of the unit disc}
\author{Julia Koch  \and Sebastian Schlei\ss inger
       \thanks{Supported by the ERC grant  ``HEVO - Holomorphic Evolution Equations'' no. 277691.}}
\date{\today}
\maketitle

\abstract{
\noindent
Let $\D$ be the unit disc and $z_0\in\D.$ We determine the value range $\{f(z_0)\,|\, f\in \mathcal{R}^\geq\}$, where $\mathcal{R}^\geq$ is the set of holomorphic functions $f:\D\to\D$ with $f(0)=0$ and $f'(0)\geq0$ that have only real coefficients in their power series expansion around $0$, and the smaller set $\{f(z_0)\,|\, f\in \mathcal{R}^\geq, \text{$f$ is typically real}\}.$\\ 
Furthermore, we describe a third value range $\{ f(z_0) \,|\, f \in \mathcal{I}\}$, where $\mathcal{I}$ consists of all univalent self-mappings of the upper half-plane $\Ha$ with hydrodynamical normalization which are symmetric with respect to the imaginary axis.\\
}



\setlength{\parindent}{0pt}

\section{Introduction}

 There are many results that determine value regions $V_A(z_0)=\{f(z_0)\,|\, f\in A\}$, where $z_0$ is a fixed point and $f$ runs through a set $A$ of analytic functions in the unit disc $\D$ or the upper half-plane $\Ha$. In this paper we describe $V_A(z_0)$ for the following cases.\\

In Section \ref{mao}, we let $A$ be the set of all univalent self-mappings of $\Ha$ with hydrodynamical normalization which are symmetric with respect to the imaginary axis. The value set $V_A(z_0)$ can be computed  by analysing the reachable set of a certain  Loewner equation.\\

In Section \ref{kissinger}, we consider the case where $A$ consists of all typically real mappings $f:\D\to\D$ with $f(0)=0$ and $f'(0)>0.$ It turns out that $V_A(z_0)$ is equal to the value set $\{f(z_0)\,|\, f\in A \text{ and $f$ is univalent}\}$, which has been determined by Prokhorov in \cite{MR1335945}. A univalent self mapping $f:\D\to\D$ with $f(0)=0$ and $f'(0)>0$ is typically real if and only if $f(\overline{z})=\overline{f(z)},$ i.e. $f$ is symmetric with respect to the real axis. 
We also consider the smaller set $\{f(z_0)\,|\, f\in A \text{ and $f'(0)=\tau$}\}$ for fixed $\tau \in (0,1].$\\
Furthermore, we determine $V_A(z_0)$ for the case where $A$ consists of all self-mappings $f:\D\to\D$ with $f(0)=0$ and $f'(0)\geq0$ that have only real coefficients in its power series expansion around $0$.\\

\textbf{Acknowledgements}: The authors would like to thank the anonymous referee for their time and effort taken to review our manuscript, and for their helpful suggestions.

\section{Hydrodynamic normalization}\label{mao}

Let $\mathcal H$ be the set of all univalent, i.e. holomorphic and injective, mappings $f:\Ha\to\Ha$ with hydrodynamic normalization at infinity, i.e. 
\begin{equation}\label{norm}f(z)=z-\frac{c}{z}+\gamma(z),\end{equation}
where $\operatorname{hcap}(f):=c\geq0$, which is usually called \emph{half-plane capacity}, and $\gamma$ satisfies $\angle\lim_{z\to \infty}z\cdot\gamma(z)=0.$

\begin{remark} Let $f\in \mathcal H$ with $\operatorname{hcap}(f)=c.$ If we transfer $f$ to the unit disc by conjugation by the Cayley transform, then we obtain a function $\tilde{f}:\D\to\D$ having the expansion
$$\tilde{f}(z)=z-\frac{c}{4}(z-1)^3+\tilde{\gamma}(z),$$
where $\angle\lim_{z\to1}\frac{\tilde{\gamma}(z)}{(z-1)^3}=0.$
\end{remark}

Let $z_0\in \Ha.$ From the boundary Schwarz lemma it is clear that $\Im(f(z_0))\geq \Im(z_0)$ and $\Im(f(z_0))=\Im(z_0)$ if and only if $f$ is the identity. It is quite simple to show that
$$V_{\mathcal H}(z_0) = \{w\in\Ha \,|\, \Im(w)>\Im(z_0)\}\cup \{z_0\},$$
see \cite{MR3262210}, Theorem 2.4.\\

Next, let $$\mathcal I = \{f\in\mathcal{H}\,|\, f(-\overline{z}) = -\overline{f(z)} \text{ for all $z\in\Ha$}\}.$$
$\mathcal I$ consists of all $f\in \mathcal H$ such that the image $f(\Ha)$ is symmetric with respect to the imaginary axis.

\begin{theorem}\label{chordal_free} Let $z_0\in \Ha.$ If $\Re(z_0)=0$, then $V_{\mathcal I}(z_0)=\{z_0+it\,|\, t\in[0,\infty)\}.$\\
Next, assume $\Re(z_0)>0$ and define the two curves $C(z_0)$ and $D(z_0)$ by
$$C(z_0)=\left\{\sqrt{z_0^2-2t}\,|\, t\in[0,\infty)\right\} =\{x+iy\in\Ha\, |\, x\cdot y=\Re(z_0)\Im(z_0),\, x\in(0,\Re(z_0)]\},$$
$$D(z_0)=\left\{z_0+e^{i\arg(z_0)}\cdot t\,|\, t\in[0,\infty)\right\}.$$
Then, the set $\overline{V_{\mathcal I}(z_0)}$ is the closed subset of $\Ha$ bounded by $C(z_0)$ and $D(z_0),$ and $V_{\mathcal I}(z_0)=\{z_0\}\cup \overline{V_{\mathcal I}(z_0)}\setminus D(z_0).$ \\
The case $\Re(z_0)<0$ follows from the case $\Re(z_0)>0$ by reflection w.r.t the imaginary axis.
\end{theorem}

The value set $f^{-1}(z_0)$ for the inverse functions is given in a quite similar way.

\begin{theorem}\label{chordal_free_inv} Let $z_0\in \Ha$ and define 
$$V_{\mathcal I}^*(z_0)=\{f^{-1}(z_0)\,|\, f\in \mathcal{I}, z_0\in f(\Ha)\}.$$
 If $\Re(z_0)=0$, then $V_{\mathcal I}^*(z_0)=\{z_0-it\,|\, t\in[0,\Im(z_0))\}.$\\
Next, assume $\Re(z_0)>0$ and define the two curves $C^*(z_0)$ and $D^*(z_0)$ by
$$C^*(z_0)=\left\{\sqrt{z_0^2+2t}\,|\, t\in[0,\infty)\right\}=\{x+iy\in\Ha\, |\, x\cdot y=\Re(z_0)\Im(z_0),\, x\in[\Re(z_0),\infty)\},$$
$$ D^*(z_0)=\left\{z_0-e^{i\arg(z_0)}\cdot t\,|\, t\in[0,|z_0|)\right\}.$$
Then, the closure $\overline{V_{\mathcal I}^*(z_0)}$ is the closed subset of $\Ha$ bounded by the curves $C^*(z_0),$ $D^*(z_0)$ and the positive real axis. The set $V_{\mathcal I}^*(z_0)$ is given by $V_{\mathcal I}^*(z_0)=\{z_0\}\cup \overline{V_{\mathcal I}^*(z_0)}\setminus (D^*(z_0) \cup [0,\infty)).$ \\
The case $\Re(z_0)<0$ follows from the case $\Re(z_0)>0$ by reflection w.r.t the imaginary axis.
\end{theorem}

Figure \ref{fig1} shows the curves $C(1+i)$ and $D(1+i)$ (dashed), as well as $C^*(1+i)$ and $D^*(1+i)$.
\begin{figure}[h]
\begin{center}

\includegraphics[width=5cm]{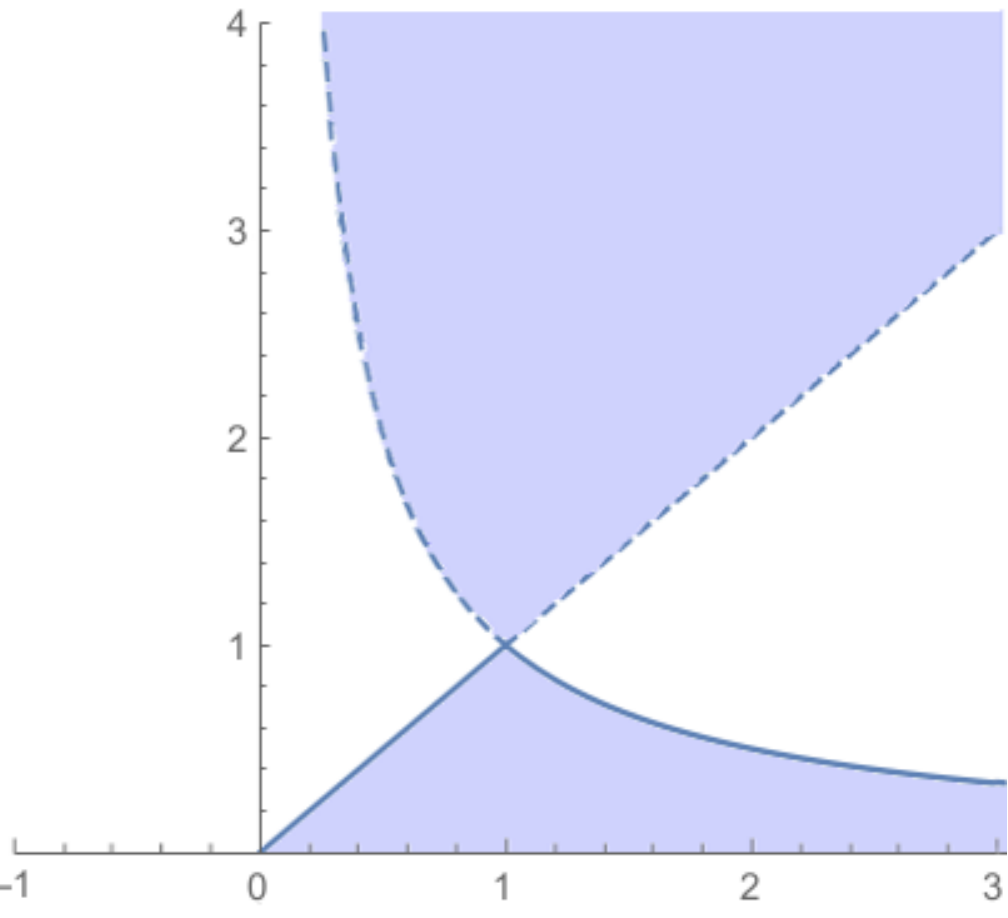}
\caption{$V_{\mathcal I}(1+i)$ and $V_{\mathcal I}^*(1+i).$}
\label{fig1}
\end{center}
\end{figure}

\begin{proof}[Proof of Theorem \ref{chordal_free}]
Without loss of generality we may assume that $z_0\in Q_1:=\{z\in\C\, |\, \Re z\geq0, \Im z>0\}.$\\

Now consider the chordal Loewner equation
\begin{equation}\label{Loewner0C}
\text{$\dot{f}_t(z) = \int_\R \frac{1}{a-f_t(z)} \, \alpha_t(da)$ for a.e. $t\geq0$, \quad $f_0(z)=z\in\Ha$,} 
\end{equation}
where $\alpha_t$ is a Borel probability measure on $\R$ for every $t\geq0$, and the function $t\mapsto \int_\R \frac{1}{a-z} \, \alpha_t(da)$ is measurable for every $z\in\Ha.$
For every $f\in \mathcal H$ there exists $T>0$ and such a family $\{\alpha_t\}_{t\geq0}$ of probability measures such that the solution $\{f_t\}_{t\geq0}$ of \eqref{Loewner0C} satisfies $f_T=f$; see \cite{MR1201130}, Theorem 5.\\

Now let $f_T=f\in \mathcal I.$ Then we can find a solution $\{f_t\}_{t\in[0,T]}$ such that $f_t\in \mathcal I$ for all $t\in[0,T],$ which means that $ \alpha_t$ can be written as $\alpha_t=1/2 \mu^*_t + 1/2 \mu_t$, where $ \mu_t$ is a probability measure supported on $[0,\infty)$ and $ \mu^*_t$ is the reflection of $ \mu_t$ to $(-\infty,0].$\\ This leads to the symmetric Loewner equation

\begin{equation}\label{general}
\dot{f}_t(z) =  \int_\R \frac{1/2}{a-f_t(z)} \,   \mu_t(da) + \int_\R \frac{1/2}{-a-f_t(z)} \, \mu_t(da)\\
=
  \int_\R \frac{f_t(z)}{a^2-f_t(z)^2}  \, \mu_t(da) =  \int_\R \frac{f_t(z)}{u-f_t(z)^2}  \, \nu_t(du),
\end{equation}
where we put $u=a^2 \in [0,\infty)$ and $\nu_t(A)=\mu_t(\sqrt{A})$, $\nu_t(B)=0$ for Borel sets $A\subset[0,\infty)$ and $B\subset(-\infty,0)$.\\

Thus we can consider the initial value problem  
\begin{equation}\label{Loewner1C} \dot{w}(t) = \int_\R \frac{w(t)}{u-w(t)^2}  \, \nu_t(du),\quad w(0)=z_0\in\Ha, \end{equation}

and have 
\begin{equation}\label{mass}
V_{\mathcal I}(z_0)=\{w(T)\,|\, \text{$w(t)$ solves \eqref{Loewner1C}}, T\geq 0\}.
\end{equation}

Next, we observe that the set $\mathcal{I}_S:=\{f\in\mathcal I \,|\, \text{$f(Q_1)=Q_1\setminus \gamma$ for a simple curve $\gamma$}\}$ is dense in $\mathcal I$ by a standard argument for univalent functions; see \cite[Section 3.2]{Duren:1983}.\\
Denote by $\delta_x$ the Dirac measure in $x\in\R.$ 
If $f\in\mathcal{I}_S,$ then we can find a continuous function $U:[0,\infty)\to [0,\infty)$ such that the measure $\nu_t=  \delta_{U(t)}$ in \eqref{general} generates $f$, i.e. $f_T=f;$ see Chapter I.3.8 in \cite{tammi}, which considers the unit disc and the corresponding symmetric radial Loewner equation; see also \cite[\S 3]{Gor86}, \cite[\S 5,6]{Gor15}. \\

Consider the corresponding initial value problem 

\begin{equation}\label{Loewner1Cslits} \dot{w}(t) = \frac{w(t)}{U(t)-w(t)^2} ,\quad w(0)=z_0\in\Ha, \end{equation}
where $U:[0,\infty)\to [0,\infty)$ is a continuous function. Denote by $R(z_0)$ the reachable set of this equation, i.e. $R(z_0):=\{w(T)\,|\, \text{$w(t)$ solves \eqref{Loewner1Cslits}}, T\geq 0\}.$ Then $V_{\mathcal{I}_S}(z_0)\subset R(z_0)$ and because of the denseness of $\mathcal{I}_S$ in $\mathcal I,$ we have

\begin{equation}\label{dense}\overline{R(z_0)} = \overline{V_{\mathcal I}(z_0)}.\end{equation}

 Now we determine the set $V_{\mathcal I}(z_0)$.\\

If $\Re(z_0)=0$, it is clear that $w(t)\in \{z_0+is\,|\, s\in[0,\infty)\}$ for all $t\in[0,\infty)$. The solution to \eqref{Loewner1Cslits} for $U(t)\equiv0$ is given by $w_1(t)=\sqrt{z_0^2-2t}.$ As $\Im(w_1(t))\to\infty$ as $t\to\infty,$ we conclude that $V_{\mathcal I}(z_0)=\{z_0+it\,|\, t\in[0,\infty)\}.$\\

Now assume that $\Re(z_0)>0.$\\

Step 1: First, we determine $R(z_0).$ We write $w(t)=\xi(t)+i \eta(t)$ and $z_0=\xi_0+i\eta_0$; thus \eqref{Loewner1Cslits} reads
$$ \dot{\xi}(t) =  \frac{\xi(t)(U(t)-|w(t)|^2)}{|U(t)-w(t)^2|^2},\qquad    
\dot{\eta}(t)=  \frac{\eta(t)(U(t)+|w(t)|^2)}{|U(t)-w(t)^2|^2},\qquad \xi(0)=\xi_0,\quad \eta(0)=\eta_0.$$
As $t\mapsto \eta(t)$ is strictly increasing, we can parametrize $\xi$ and $U$ by $\eta$ and obtain
\begin{align}\label{nacheta}
\frac{d\xi}{d\eta} = \frac{\xi(\eta)}{\eta}\frac{U(\eta)-|w(\eta)|^2}{U(\eta)+|w(\eta)|^2}.
\end{align}

Since $U(\eta)\geq0$, we have
$$ -1 \leq \frac{U(\eta) - |w(\eta)|^2}{U(\eta) + |w(\eta)|^2} < 1$$
which yields the inequalities
$$\frac{-\xi}{\eta}\leq  \frac{d\xi}{d\eta} < \frac{\xi}{\eta}.$$
By solving the equations $\frac{d\xi'}{d\eta}= -\frac{\xi'}{\eta}$ and $\frac{d\xi'}{d\eta}= \frac{\xi'}{\eta}$ with $\xi'(\eta_0)=\xi_0,$ we arrive at 
\begin{equation}\label{thunder}\frac{\eta_0}{\eta(t)}\leq \frac{\xi(t)}{\xi_0} < \frac{\eta(t)}{\eta_0}
\end{equation}
for all $t>0$. We have equality for the left case when $U(t)\equiv 0$, which leads to the solution $w(t)=\sqrt{z_0^2-2t}$, i.e. the curve $C(z_0)$. The case $\frac{\xi}{\xi_0} = \frac{\eta}{\eta_0}$ corresponds to the curve $D(z_0)$, which does not belong to the set $R(z_0)\setminus\{z_0\}.$\\

On the Riemann sphere $\hat{\C}$, the two curves $\hat{C}(z_0)=C(z_0)\cup\{\infty\}$ and $\hat{D}(z_0)=D(z_0)\cup\{\infty\}$ intersect at $z_0$ and $\infty$, and form the boundary of two Jordan domains. We denote by $J(z_0)$ the closure of the one that is contained in $\Ha\cup \{\infty\}.$  Note that $R(z_0)\subset J(z_0)$ by \eqref{thunder}.
We wish to show that  $R(z_0)=\{z_0\}\cup J(z_0)\setminus \hat{D}(z_0).$ \\

To this end, consider \eqref{nacheta} with the driving term 
$$U(\eta)=\frac{1+x}{1-x}\left(\xi_0\left(\frac{\eta}{\eta_0}\right)^{2x}+\eta^2\right), \quad -1\leq x<1.$$ 
This yields
$$\frac{U(\eta) - |w(\eta)|^2}{U(\eta) + |w(\eta)|^2}\equiv x,$$
and thus
$$\xi(\eta)=\xi_0\left(\frac{\eta}{\eta_0}\right)^x,$$
and it is easy to see that 
$$J(z_0)\setminus\hat D(z_0)=\left\{\xi_0\left(\frac{\eta}{\eta_0}\right)^x\,\Big|\, \eta \in[0,\infty),\,  -1\leq x<1\right\}\subseteq R(z_0).$$

Step 2: Finally, we show that $V_{\mathcal I}(z_0) = \{z_0\}\cup J(z_0)\setminus \hat{D}(z_0),$ which concludes the proof.\\
 As we already know that $\overline{R(z_0)} = \overline{V_{\mathcal I}(z_0)}$ (equation \eqref{dense}), we only need to prove that $\hat{D}(z_0)\setminus\{z_0\}$ has empty intersection with $V_{\mathcal I}(z_0).$ \\
Recall \eqref{mass} and let $w(t)$ be a solution to \eqref{Loewner1C}. We write again $w(t)=\xi(t)+i\eta(t).$ Then 
 
$$ \dot{\xi} = \xi(t) \int_\R \frac{u-|w(t)|^2}{|u-w(t)^2|^2}\,\mu_t(du) < \xi(t) \int_\R \frac{u+|w(t)|^2}{|u-w(t)^2|^2}\,\mu_t(du),\qquad   
\dot{\eta} =  \eta(t)\int_\R\frac{u+|w(t)|^2}{|u-w(t)^2|^2}\,\mu_t(du).$$
Again, $t\mapsto \eta(t)$ is strictly increasing, and we parametrize $\xi$ by $\eta$ to get
$$\frac{d\xi}{d\eta} < \frac{\xi}{\eta},$$ which yields $\frac{\xi(t)}{\xi_0}<\frac{\eta(t)}{\eta_0}$ for all $t>0,$ hence $(D(z_0)\setminus\{z_0\})\cap V_{\mathcal I}(z_0) = \emptyset.$

\end{proof}

The proof of Theorem \ref{chordal_free_inv} is  completely analogous.

\section{Interior normalization}\label{kissinger}

\subsection{Univalent self-mappings with real coefficients}

The classical Schwarz lemma tells us that a holomorphic map $f:\D\to \D$ with $f(0)=0$ has the property $|f(z)|\leq |z|$ for every $z\in\D.$ A refinement of this result was obtained by Rogosinski \cite{rogo2} (see also \cite{Duren:1983}, p. 200) who determined the value range $V_A(z_0)$ where $z_0\in \D$ and $A$ consists of all holomorphic self-mappings $f$ of $\D$ that fix the origin and $f'(0)>0$. In \cite{MR3262210} the authors consider the set of all univalent self-mappings of $\D$ with the same normalization, i.e. 
$$A=\mathcal S_>:=\{f:\D\to\D \text{ univalent}, f(0)=0, f'(0)>0\}.$$
One can obtain a refinement of the main result of \cite{MR3262210} by considering the normalization $f'(0)=e^{-T}$, $T>0,$ instead of $f'(0)>0$; see \cite{MR3398791}. \\

Let $\mathcal U$ be the set of all $f\in \mathcal{S}_>$ in $\D$ having only real coefficients in their Taylor expansion around the origin. 
The following result has been proven in \cite{MR1335945}. The proof uses Pontryagin's maximum principle, which is applied to the radial Loewner equation.\\
An elementary proof of the theorem is given in \cite{Pfrang}. 

\begin{theorem}[\cite{MR1335945}]\label{radial_free}\label{pro} Let $z_0\in\D\setminus\{0\}.$ \\
If $z_0 \in \R,$ then $V_{\mathcal U}(z_0) \cup \{0\}$ is the closed interval with endpoints $0$ and $z_0$.\\Define the two curves $C_+(z_0)$ and $C_-(z_0)$ by
\begin{align*}
C_{+}(z_0):=\left\{\frac1{2z_0}(e^t (z_0+1)^2 -2z_0- 
 e^{t/2}(z_0 + 1) \sqrt{e^t (z_0 + 1)^2-4z_0 })\,|\,  t\in[0,\infty]\right\},\\
C_{-}(z_0):=\left\{\frac1{2z_0}(e^t (z_0-1)^2 + 2 z_0 +
 e^{t/2} (z_0 - 1) \sqrt{e^t (z_0 - 1)^2 + 4z_0})\,|\,  t\in[0,\infty]\right\}.
\end{align*}

If $z_0 \not\in\R,$ then $V_{\mathcal U}(z_0)\cup\{0\}$ is the closed region whose boundary consists of the two curves $C_+(z_0)$ and $C_-(z_0)$, which only intersect at $t\in\{0,\infty\}.$\\

Furthermore, for $z_0 \not\in\R,$ any boundary point of $V_{\mathcal U}(z_0)$ except $0$ can  be reached by only one mapping $f\in \mathcal{R}$, which is of the form $f_{1,t}(z)= \frac1{2z}(e^t (z+1)^2 -2z- 
 e^{t/2}(z + 1) \sqrt{e^t (z + 1)^2-4z })$ or $f_{2,t}(z)= \frac1{2z}(e^t (z-1)^2 + 2 z +
 e^{t/2} (z - 1) \sqrt{e^t (z - 1)^2 + 4z})$ with $t\in[0,\infty)$.\\
 The mapping $f_{1,t}$ maps $\D$ onto $\D\setminus[2 e^t -1- 2 e^{t/2} \sqrt{e^t-1}, 1]$ and $f_{2,t}$ maps $\D$ onto $\D\setminus[-1, -2 e^t +1+ 2 e^{t/2} \sqrt{e^t-1}].$
\end{theorem}

\begin{remark} For $\tau\in(0,1]$ let $\mathcal{U}(\tau)=\{f\in\mathcal U \,|\, f'(0)=\tau\}.$ The value set $V_{\mathcal{U}(\tau)}$ is described in \cite{PS16}.
\end{remark}

The value set for the inverse functions can be obtained quite similarly.

\begin{theorem}\label{radial_free_inv} Let $z_0\in\D\setminus\{0\}$ and define
$$V_{\mathcal U}^*(z_0) = \{f^{-1}(z_0) \,|\, f\in \mathcal{U}, z_0\in f(\D)\}.$$ \\
If $z_0 \in (0,1),$ then $V_{\mathcal U}^*(z_0)=[z_0,1)$, and if $z_0\in(-1,0)$, then $V_{\mathcal U}^*(z_0)=(-1,z_0]$. \\
Define the two curves $C_+^*(z_0)$ and $C_-^*(z_0)$ by
\begin{align*}
C_{+}^*(z_0):=\left\{\frac1{2z_0}(e^t (z_0+1)^2 -2z_0- 
 e^{t/2}(z_0 + 1) \sqrt{e^t (z_0 + 1)^2-4z_0 })\,|\,  t\in[-\infty,0]\right\},\\
C_{-}^*(z_0):=\left\{\frac1{2z_0}(e^t (z_0-1)^2 + 2 z_0 +
 e^{t/2} (z_0 - 1) \sqrt{e^t (z_0 - 1)^2 + 4z_0})\,|\,  t\in[-\infty,0]\right\}.
\end{align*}

Now let $\Im(z_0)>0.$ Then $\overline{V_{\mathcal U}(z_0)}$ is the closed region bounded by the curves $C_+(z_0)$, $C_-(z_0)$ and $E:=\partial\D \cap \overline{\Ha}.$ The set $V_{\mathcal U}(z_0)$ is given by $V_{\mathcal U}(z_0)=\overline{V_{\mathcal U}(z_0)}\setminus E.$
\end{theorem}

 Figure \ref{fig2} shows the set $V_{\mathcal U}(z_0)$ (orange), which lies inside the heart-shaped set
$V_{\mathcal{\mathcal S_>}}(z_0)$ (blue) that is determined in \cite{MR3262210}, and the set $V_{\mathcal U}^*(z_0)$ (red, dashed) for $z_0=0.9 e^{i\pi/4}$.

\begin{figure}[h]
\rule{0pt}{0pt}

\centering
\includegraphics[width=7cm]{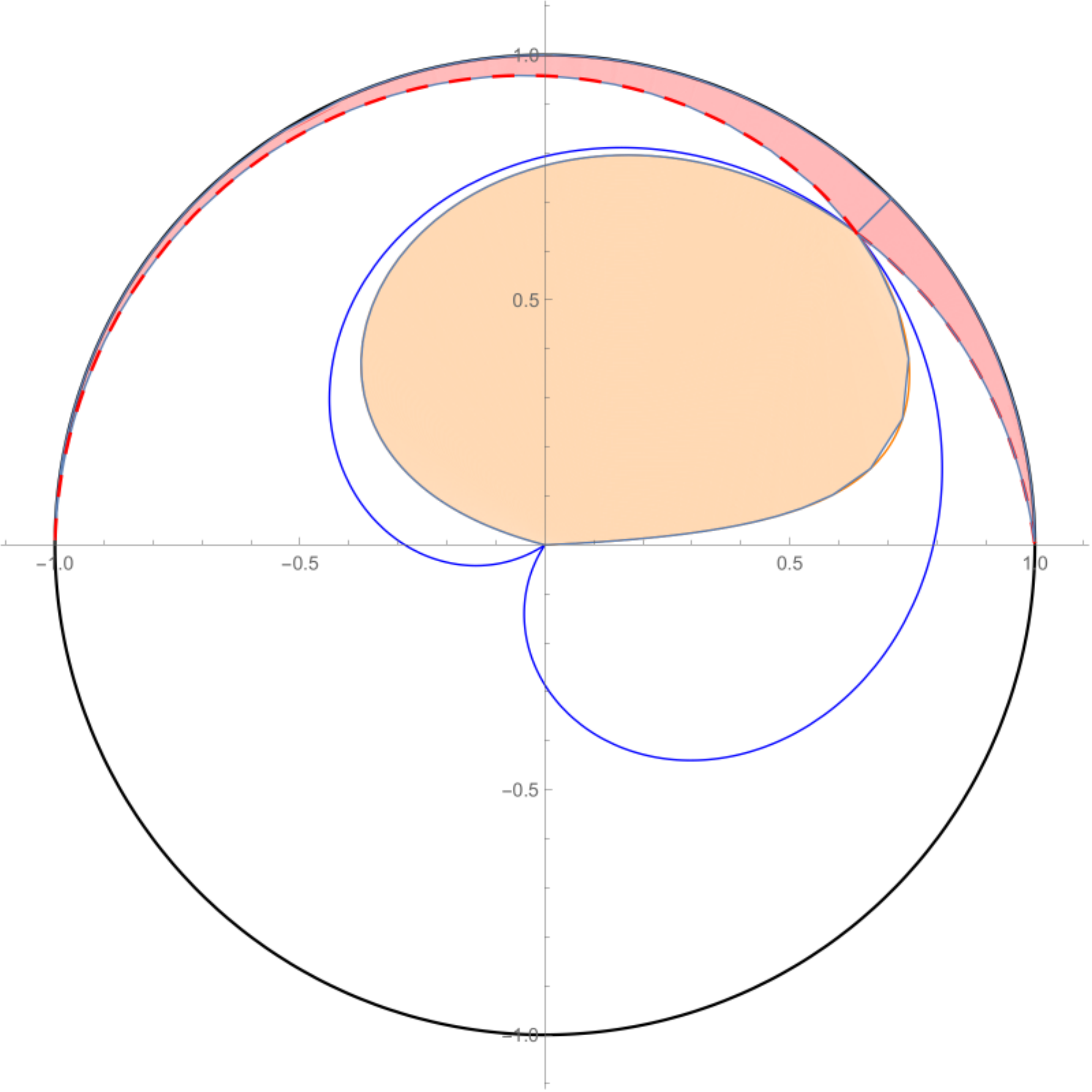}
\caption{$V_{\mathcal U}(0.9 e^{i \pi/4})$}
\label{fig2}

\end{figure}
\subsection{Bounded typically real mappings}

Following Rogosinski \cite{rogo1}, a holomorphic map $f:\D\to\C$ is called \emph{typically real} if 
$$\Im(f(z)) \Im(z) \geq0\quad \text{for all $z\in\D.$}$$
Now we define $\mathcal T$ as the set of all typically real self-mappings $f:\D\to\D$ with $f(0)=0$ and $f'(0)>0$. 
Obviously, $\mathcal U \subset \mathcal T.$ \\
From Rogosinski's work one immediately obtains an integral representation for typically real mappings, see also \cite{Rob}, Section 2. In order to determine the value region $V_{\mathcal T}(z_0)$, we will need the following integral representation for bounded typically real mappings.

\begin{theorem}[\cite{Szapiel}, Theorem 2.2]\label{szapiel} Let $f \in \mathcal T$ with $f'(0)=\tau>0.$ Then there exists a probability measure $\mu$ supported on $B:=\{(x,y)\in\R^2 \,|\, -1\leq x \leq 2\tau-1\leq y\leq1\}$ such that $$ f(z) = \frac{\sqrt{g_{\mu,\tau}(z)}-1}{\sqrt{g_{\mu,\tau}(z)}+1}, $$
where we take the holomorphic branch of the square root with $\sqrt{1}=1$ and $$g_{\mu,\tau}(z) = \int_{B} \frac{(1 + z)^2 (1 - 2 (1 - 2 \tau + x + y) z + z^2)}{(1 - 2 x z + 
   z^2) (1 - 2 y z + z^2)} \,\mu(dxdy).$$
\end{theorem}
\begin{remark}\label{rmk2} In order to show $\mathcal U \subsetneq \mathcal T$ we find a function $f_0\in\mathcal T\setminus \mathcal U$ as follows: Let $\tau=1/2$ and let $\mu$ be the point measure in $(\tau - 1, \tau).$ Then we obtain 
$$f_0(z)=\frac{\sqrt{((1 + z)^2 (1 + z^2))/(1 + z^2 + z^4)} - 
   1}{\sqrt{((1 + z)^2 (1 + z^2))/(1 + z^2 + z^4)} + 1}.$$
	The derivative $f_0'(z)$ has a zero at $z=-(1/4) + (i \sqrt{3})/4 + 1/2 \sqrt{-(9/2) - (i \sqrt{3})/2}\in\D.$ Hence, $f_0\not\in \mathcal U.$
\end{remark}

\begin{theorem}\label{Winnetou} Let $z_0\in\D\setminus\{0\}$ and $\tau \in(0,1].$ Define $\mathcal{T}(\tau):=\{f\in\mathcal T\,|\, f'(0)=\tau\}.$ \\
The set $V_{\mathcal{T}(\tau)}(z_0)$ is the image of the closed region bounded by the two circular arcs 
$$\left\{1 + \frac{4 \tau z_0}{1 - 2 y z_0 + z_0^2}\,|\, y\in[2\tau-1, 1]\right\} \; \text{and} \; \left\{\frac{(z_0+1)^2 (1 + z_0 (-4 + 4 \tau - 2 x + z_0))}{(z_0-1)^2 (1 - 2 x z_0 + 
   z_0^2)}\,|\, x\in[-1,2\tau-1]\right\}$$
	under the map $w\mapsto \frac{\sqrt{w}-1}{\sqrt{w}+1}.$
\end{theorem}
\begin{proof}
Fix some $\tau>0$. First we show that the set $A(\tau):=\{g_{\mu,\tau}(z_0)\,|\, \text{$\mu$ is a point measure on $B$}\}$ is convex.
To this end, we evaluate the function $(x,y)\mapsto s(x,y):=\frac{(1 + z)^2 (1 - 2 (1 - 2 \tau + x + y) z + z^2)}{(1 - 2 x z + 
   z^2) (1 - 2 y z + z^2)}$ on $\partial B:$
\begin{itemize}
\item[(i)] $s(x,y)=\frac{(1 + z_0)^2}{(z_0+1)^2 - 4 \tau z_0 }$ when $x=2\tau-1.$
\item[(ii)] $s(x,y)=\frac{(1 + z_0)^2}{(z_0+1)^2 - 4 \tau z_0 }$ when $y=2\tau-1.$ 
\item[(iii)] $s(x,y)=1 + \frac{4 \tau z_0}{1 - 2 y z_0 + z_0^2}$ when $x=-1.$
\item[(iv)] $s(x,y)=\frac{(z_0+1)^2 (1 + z_0 (-4 + 4 \tau - 2 x + z_0))}{(z_0-1)^2 (1 - 2 x z_0 + 
   z_0^2)}$ when $y=1$.
\end{itemize}

We see that $s(\partial B)$ consists of two circular arcs connecting the points
\begin{equation}\label{lc} P(\tau) = \frac{(1 + z_0)^2}{(1+z_0)^2 - 4 \tau z_0} \quad \text{and} \quad  
Q(\tau) = 1 + \frac{4 \tau z_0}{1 - 2 z_0 + z_0^2},
\end{equation}

and the two arcs are given by 
\begin{align*}
s_{1,\tau}&:[2\tau-1,1]\to\C,\quad s_{1,\tau}(y):=1 + \frac{4 \tau z_0}{1 - 2 y z_0 + z_0^2},\\
s_{2,\tau}&:[-1,2\tau-1]\to\C, \quad s_{2,\tau}(x):=\frac{(z_0+1)^2 (1 + z_0 (-4 + 4 \tau - 2 x + z_0))}{(z_0-1)^2 (1 - 2 x z_0 + 
   z_0^2)}. 
\end{align*}
Without loss of generality we restrict to the case $\Im(z_0)\geq 0.$\\
A short calculations shows that then
\begin{align*}
\frac d{dy} \arg \frac{d}{dy} s(x,y)&=4\Im z_0 \frac{1-|z_0|^2}{|1-2 y z_0 +z_0^2|^2} \geq 0,\\
\frac d{dx} \arg\frac{d}{dx} s(x,y)&=4\Im z_0 \frac{1-|z_0|^2}{|1-2 x z_0 +z_0^2|^2} \geq 0.\\
\end{align*}
Furthermore, we have 
\begin{align*}
s_x(x,y):=\frac d{dx}s(x,y)&=\frac{4(2\tau-1-y)z^2(1+z)^2}{(1-2xz+z^2)^2(1-2yz+z^2)},\\
s_y(x,y):=\frac d{dy}s(x,y)&=\frac{4(2\tau-1-x)z^2(1+z)^2}{(1-2yz+z^2)^2(1-2xz+z^2)},
\end{align*}
and thus  
\begin{align}\label{diana}
\Im \frac{s_x(x,y)}{s_y(x,y)}=\frac{2(y-(2\tau-1))}{2\tau-1-x}\frac{1-|z|^2}{|1-2xz+z^2|^2}\Im(z)\cdot (y-x)\geq0
\end{align}
for all $(x,y)\in B.$ \\
This shows that any parallel to either the $x$- or the $y$-axis within $B$ is mapped onto a convex curve, and that whenever we map a path that points inwards in $B$, the image also lies in the interior of $s(\partial B)$. 
Therefore, the convex closure of $A(\tau)$ is equal to the set $W(\tau)$ defined as the compact region bounded by the curves $s_{1,\tau}$ and $s_{2,\tau}$.\\ From Theorem \ref{szapiel} it follows that $V_{\mathcal{T}(\tau)}(z_0)$ is the image of $K(\tau):=\{g_{\mu,\tau}(z_0)\,|\, \text{$\mu$ prob. meas. on $B$}\}$ under the map $w\mapsto \frac{\sqrt{w}-1}{\sqrt{w}+1}.$ The set $K(\tau)$ is the closure of the convex hull of $A(\tau),$ i.e. $K(\tau)=W(\tau)$, which concludes the proof.

\end{proof}

Figure \ref{fig4} shows the set $V_{\mathcal T(\tau)}(z_0)$ for $z_0=0.9 e^{i \pi/4}$ and $\tau=0.1,0.5,0.9$ (shaded regions).

\begin{figure}[h]
\centering
\includegraphics[width=7cm]{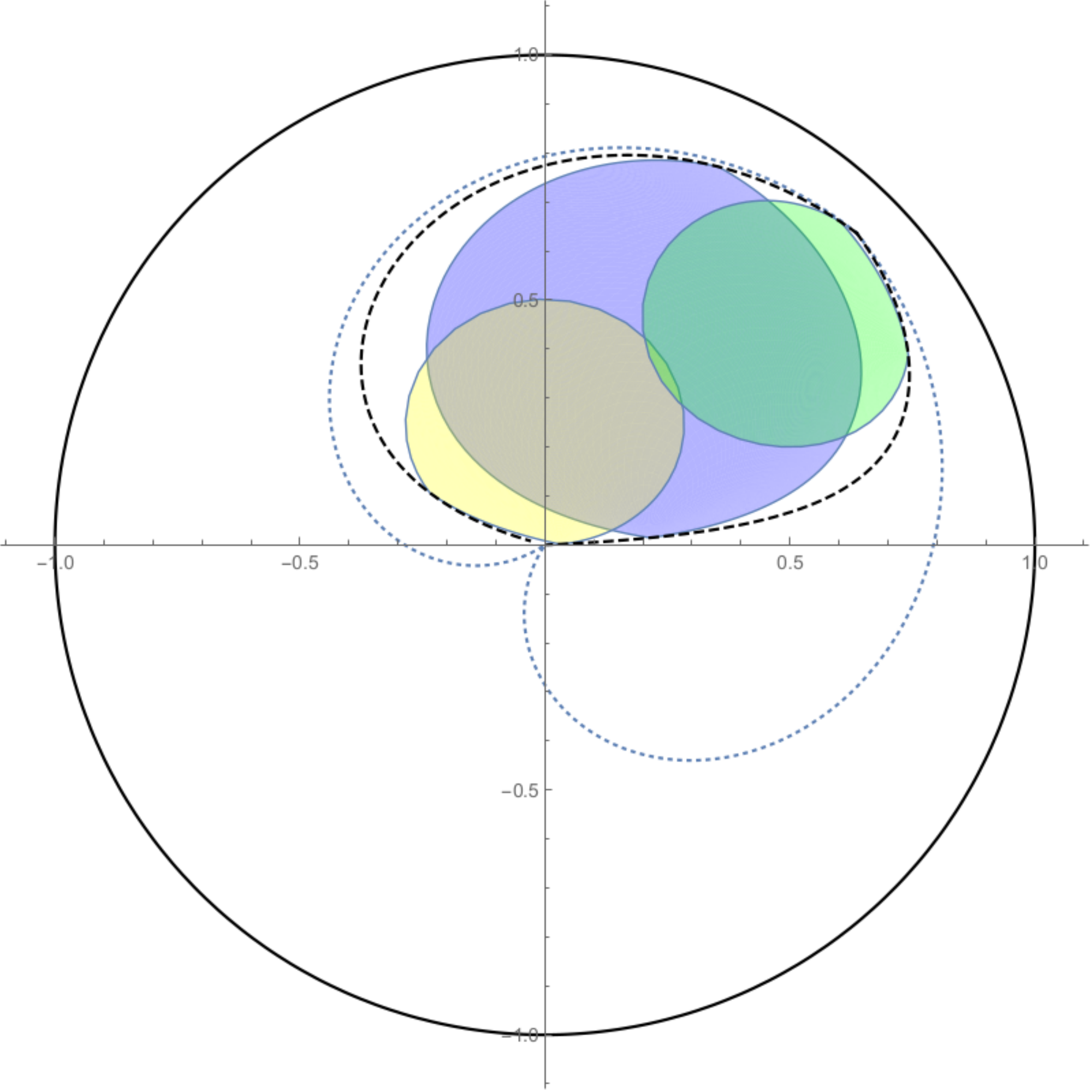}
\caption{$V_{\mathcal T(\tau)}(z_0)$.}
\label{fig4}
\end{figure}

\newpage

\begin{corollary}\label{espresso}Let $z_0\in \D\setminus\{0\}.$ Then $V_{\mathcal T}(z_0)=V_{\mathcal U}(z_0).$ \\
Furthermore, if $z_0\not\in\R$, then each $w\in\partial V_{\mathcal{T}}(z_0)$ except $0$ can be reached by only one mapping  $f\in \mathcal{T}$, which is of the form $f_{1,t}$ or $f_{2,t}$ from Theorem \ref{radial_free}.
\end{corollary}
\begin{proof}
Again, we denote by $W(\tau)$ the image of $V_{\mathcal{T}(\tau)}(z_0)$ under the map $z\mapsto (1+z)^2/(1-z)^2,$ the inverse function of $w\mapsto \frac{\sqrt{w}-1}{\sqrt{w}+1},$ which is the convex region bounded by the circular arcs $s_{1,\tau}(y),$ $y\in[2\tau-1,1],$ and $s_{2,\tau}(x),$ $x\in[-1, 2\tau-1].$\\ 
Consider the convex region $R$ bounded by the circular arc 
\begin{equation}\label{star}\tag{$*$}C:=\left\{P(\tau)\,|\, \tau\in[0,1]\right\} \text{ and the line segment }
L:=\left\{Q(\tau)\,|\, \tau\in[0,1]\right\},\end{equation}
where $Q(\tau)$ and $P(\tau)$ are defined as in \eqref{lc}.\\

Fix $\tau \in (0,1].$ To show that $W(\tau)$ is contained in $R$, assume the opposite. Then the boundary of $W(\tau)$ has to intersect either $C$ or $L$ in some other point besides $P(\tau)$ and $Q(\tau).$ 
However, it is easy to see that each of the following four equations
\begin{eqnarray*} s_{1,\tau}(y) &= P(t), \quad y,t \in \R, \\
 s_{1,\tau}(y) &= Q(t), \quad y,t \in \R, \\
 s_{2,\tau}(x) &= P(t), \quad x,t \in \R, \\
 s_{2,\tau}(x) &= Q(t), \quad x,t \in \R,
\end{eqnarray*}
has only one solution, namely $(y, t)=(1, \tau),$ $(y, t)=(2\tau-1, \tau),$ $(x, t)=(-1, \tau),$ and $(x, t)=(2\tau-1, \tau)$, respectively. (In all four cases, the second intersection point between the circles/lines is given by the limit cases $y\to \infty$ and $x\to\infty$ respectively.)\\

Hence, we have $W(\tau) \subset R$ for every $\tau\in(0,1].$ Finally, it is clear that every point contained in $R\setminus\{1\}$ (note that $P(0)=Q(0)=1$) is contained in some $W(\tau)$: since every $W(\tau)$ is convex,  the line segment between $P(\tau)$ and $Q(\tau)$ is always contained in $W(\tau)$. \\
Consequently,  $\cup_{\tau\in(0,1]} W(\tau) = R\setminus \{1\}.$ \\

Now we apply the function $z\mapsto (1+z)^2/(1-z)^2,$ the inverse function of $w\mapsto \frac{\sqrt{w}-1}{\sqrt{w}+1},$ to the curves $C_+(z_0)$ and $C_-(z_0)$ from Theorem \ref{pro} and we obtain the curves 
$$\left\{\frac{(1 + z_0)^2}{(1 + z_0)^2 - 4 e^{-t} z_0}\,|\, t \in[0,\infty]\right\} \text{ and } 
\left\{1 + \frac{4 e^{-t} z_0}{(z_0-1)^2}\,|\, t \in[0,\infty]\right\},$$ which are the very same curves as \eqref{star}. Thus, we conclude that $V_{\mathcal T}(z_0)=V_{\mathcal U}(z_0)$.\\

Finally, assume $z_0\not\in\R$ and let $w\in\partial V_{\mathcal T}(z_0)\setminus\{0\}.$ Then $w=P(\tau)$ or $w=Q(\tau)$ for a unique $\tau\in(0,1]$ and the proof of Theorem \ref{Winnetou} shows that there is only one mapping $f\in V_{\mathcal{T}(\tau)}(z_0)$ with $f(z_0)=w.$ From Theorem \ref{radial_free} it follows that $f$ is of the form $f_{1,t}$ or $f_{2,t}.$
\end{proof}

\subsection{Real coefficients}

Finally we take a look at one further value region, determined by Rogosinski in \cite{rogo2}, p.111: Let $\mathcal{R}$ be the set of all holomorphic functions $f:\D\to\D$ with $f(0)=0$ that have only real coefficients in the power series expansion around 0. Then $V_{\mathcal{R}}(z_0)$ is the intersection of the two closed discs whose boundaries are the circles through $1, z_0, -z_0$ and through $-1, -z_0, z_0$ respectively.\\

Let $\mathcal{R}^\geq$ be the set of all holomorphic functions $f\in \mathcal{R}$ with $f'(0)\geq 0$ and $z_0\in\D\setminus\{0\}.$ Then we have $$V_\mathcal{T}(z_0) \subset V_{\mathcal{R}^\geq}(z_0) \subset V_{\mathcal{R}}(z_0).$$ 
It is clear that $V_{\mathcal{R}^\geq}(z_0)\not= V_{\mathcal{R}}(z_0)$ as the point $-z_0$ belongs to $V_{\mathcal{R}}(z_0)$ and there is only one mapping $f\in \mathcal{R}(z_0)$ with $f(z_0)=-z_0$, namely $f(z)=-z$ for all $z\in\D.$ \\
Furthermore, if $z_0\not\in\R,$ we have $V_{\mathcal T}(z_0)\subsetneq V_{\mathcal{R}^\geq}(z_0)$ which can be seen as follows: The boundary points of $V_{\mathcal{R}}(z_0)$ can be reached only by the functions $z\mapsto \pm z \frac{z-x}{zx-1},$ $x\in[-1,1],$ see \cite{rogo2}, p.111. Hence, by Corollary \ref{espresso} we have $\partial V_{\mathcal T}(z_0)\cap \partial V_{\mathcal{R}}(z_0) = \{z_0\}.$ For $0<x<1$, the function $f(z)= z\frac{z-x}{zx-1}$ satisfies $f'(0)=x>0$ and $f(z_0)\not=z_0.$ This gives us $z_0\frac{z_0-x}{z_0x-1}\in V_{\mathcal{R}^\geq}(z_0) \setminus V_{\mathcal T}(z_0).$\\

\begin{theorem}\label{Dost} Let $z_0\in \D\setminus\{0\}.$ Then $V_{\mathcal{R}^\geq}(z_0)$ is the closed convex region bounded by the following three curves:
\begin{eqnarray*}
	A &=& \left\{z_0\frac{z_0-x}{z_0x-1} \,\big|\, x\in[0,1]\right\}, \quad
	B = \left\{z_0\frac{z_0+x}{z_0x+1} \,\big|\, x\in[0,1]\right\},\\
	C &=& \left\{\frac{z_0^2 (z_0+2x-1)}{1+2xz_0-z_0
} \,\big|\, x\in[0,1]\right\}.
\end{eqnarray*}
\end{theorem}
\begin{proof}
Let $f\in \mathcal{R}^\geq.$ Then $g(z):= (1+f(z))/(1-f(z))$ maps $\D$ into the right half-plane with $g(0)=1$, $g$ has only real coefficients in its power series expansion around $0$ and $g'(0)=2f'(0)\geq0.$ 
Due to the Herglotz representation (\cite{Duren:1983}, Section 1.9) we can write $g$ as
\begin{equation}
g(z) = \int_{\partial \D} \frac{u+z}{u-z}\, \nu(du),
\label{eq:1}
\end{equation}
for some probability measure $\nu$ on $\partial \D.$\\
As $g$, and thus $\nu$, is symmetric with respect to the real axis, one can rewrite \eqref{eq:1} as
\begin{equation}
g(z) = \int_{\partial \D} 1/2 \frac{u+z}{u-z} + 1/2 \frac{\overline{u}+z}{\overline{u}-z}\, \nu(du) = 
 \int_{\partial \D} \frac{1-z^2}{1 - 2 z \text{Re}(u) + z^2}\, \nu(du),
\end{equation}
or
\begin{equation}\label{Myschkin}
g(z) = G_\mu(z):=\int_{[0,\pi]} \frac{1-z^2}{1 - 2 z \cos(u) + z^2}\, \mu(du),
\end{equation}
where $\mu$ is a probability measure on $[0,\pi]$ that additionally fulfils
\begin{equation}\label{workhard}\int_{[0,\pi]} \cos(u) \,\mu(du) \geq  0,\end{equation}
as the last integral is equal to $g'(0)/2=f'(0).$ Thus we need to determine the extreme points of the convex set of all measures on $[0,\pi]$ that satisfy \eqref{workhard}. According to \cite{Winkler}, Theorem 2.1, 
this set of extreme points is contained in the set $S$ that consists of all point measures  $\mu=\delta_\phi$ on $[0,\pi]$ satisfying \eqref{workhard} and of all convex combinations $\mu= \lambda \delta_\phi + (1-\lambda)\delta_\varphi$, with $\phi,\varphi\in[0,\pi]$, $\phi\not=\varphi,$ $\lambda\in(0,1),$ satisfying \eqref{workhard}.\\

We can now determine $V_{\mathcal{R}^\geq}(z_0)$ as follows: Denote by $W_{\mathcal{R}^\geq}(z_0)$ the image of $V_{\mathcal{R}^\geq}(z_0)$ under the injective map $w\mapsto (1+w)/(1-w).$ Then $W_{\mathcal{R}^\geq}(z_0)$ is the closure of the convex hull of the set $\{G_\mu(z_0) \,|\, \mu \in S\}.$\\

The point measures from $S$ are, of course, all $\delta_\phi$ with $\phi\in[0,\pi/2],$ and  \eqref{Myschkin} gives us the curve 
\begin{equation}\label{Acurve}\frac{1-z_0^2}{1 - 2 z_0 x + z_0^2}, \quad x\in[0,1],
\end{equation}
a circular arc connecting the points $\frac{1-z_0^2}{1+z_0^2}$ and $\frac{1+z_0}{1-z_0}.$ This curve is the image of $A$ under the map $w\mapsto\frac{1+w}{1-w}.$ The image of $B$ is the line segment
\begin{equation}\label{Bcurve}\frac{1 + 2 x z_0 + z_0^2}{1 - z_0^2}, \quad x\in[0,1],
\end{equation}
which connects $\frac{1 + z_0^2}{1 - z_0^2}$ to $\frac{1+z_0}{1-z_0}.$\\

The other measures have the form $\lambda \delta_\phi + (1-\lambda)\delta_\varphi$ with w.l.o.g. $\phi\in[0,\pi/2],$ $\varphi\in[0,\pi]$, $\phi\not=\varphi$, $\lambda\in(0,1)$ such that $\lambda \cos(\phi)+(1-\lambda)\cos(\varphi) \geq 0.$ They lead to the set
\begin{equation}\label{Set}\lambda \frac{1-z_0^2}{1 - 2 z_0 x + z_0^2} + (1-\lambda) \frac{1-z_0^2}{1 - 2 z_0 y + z_0^2} , \quad x\in[0,1], y\in[-1,1], \lambda \in (0,1), \lambda x + (1-\lambda) y \geq 0.\end{equation}

If we take $x=1,$ $y\in[-1,0]$ and $\lambda=\frac{y}{y-1}$ (which is equivalent to $\lambda  + (1-\lambda) y = 0$), then we obtain

\begin{equation}\label{Ccurve} \frac{(1 + z_0) (1 - 2 (1 + y) z_0 + z_0^2)}{(1 - z_0) (1 - 2 y z_0 + z_0^2)}. 
\end{equation}
The above expression describes a circular arc connecting the points $\frac{1 + z_0^2}{1 - z_0^2}$ and $\frac{1-z_0^2}{1+z_0^2}$, and this arc is the image of the curve $C$ under the map  $w\mapsto\frac{1+w}{1-w}.$ Note that the point $\frac{1+z_0}{1-z_0}$ lies on the full circle, which is obtained when $y\to\infty.$ \\Denote by $\Delta$ the closed bounded region bounded by the three curves \eqref{Acurve}, \eqref{Bcurve} and \eqref{Ccurve}. Then $\Delta$ is convex. We are done if we can show that for all other measures from \eqref{Set}, $G_\mu(z_0)$ belongs to $\Delta.$ Then we can conclude that $\Delta$ is the closure of the convex hull of $\{G_\mu(z_0) \,|\, \mu\in S\}.$\\

First, we consider the points from \eqref{Set} for $x\in[0,1],$ $y\in[-1,0]$ and again $\lambda=\frac{y}{y-x}:$
\begin{itemize}
\item[a)]
 For $y=1,$ we obtain the curve
\begin{equation}\label{Incurve}
 -\frac{(-1 + z_0) (1 + z_0 (2 - 2 x + z_0))}{(1 + z_0) (1 - 2 x z_0 + z_0^2)} , \quad x\in[0,1].
\end{equation}
Like the curve \eqref{Ccurve}, this arc also connects $\frac{1 + z_0^2}{1 - z_0^2}$ and $\frac{1-z_0^2}{1+z_0^2}$. For $x\to\infty$ we obtain $\frac{1-z_0}{1+z_0}$, which is the second intersection point of the full circles corresponding to \eqref{Acurve} and \eqref{Bcurve} (note that this point is mapped onto $-z_0$ under $w\mapsto\frac{1+w}{1-w}$). We conclude that \eqref{Incurve} is contained in $\Delta.$ The convex set bounded by \eqref{Ccurve} and \eqref{Incurve} will be denoted by $\Delta_0$, and is, of course, contained in $\Delta.$ (Figure \ref{fig6} shows the image of $\Delta_0$ under the map $w\mapsto\frac{1+w}{1-w}$.)
\item[b)] Now fix $x\in[0,1)$ and \eqref{Set} becomes  
\begin{equation}\label{Incurve2}
 \frac{(-1 + z_0^2) (1 - 2 (x + y) z_0 + z_0^2)}{(-1 + 2 x z_0 - z_0^2) (1 - 2 y z_0 + 
   z_0^2)} , \quad y\in[0,1].
\end{equation}
This set is a circular arc connecting $\frac{1-z_0^2}{1+z_0^2}$ ($y=0$) to a point on \eqref{Incurve} ($y=1$). For $y\to \infty$ we obtain the point $\frac{1 - z_0^2}{1 - 2 x z_0 + z_0^2}$, which lies on \eqref{Acurve} and is not contained in $\Delta_0.$ We conclude that the arc \eqref{Incurve2} lies within $\Delta_0$ and thus in $\Delta.$
\end{itemize}

Finally, assume there are $x\in[0,1], y\in[-1,1], \lambda\in(0,1)$ such that $\lambda x + (1-\lambda) y > 0$ (i.e. $G_\mu'(0)>0$) and the corresponding point \eqref{Set} lies outside $\Delta.$ 
Since it has nevertheless to lie in $V_{\mathcal{R}}(z_0)$, the line segment between this point and $\frac{1+z_0}{1-z_0}$ must intersect the curve \eqref{Ccurve}. Thus, the set $\{f(z_0) \,|\, f\in \mathcal{R}, f'(0)>0 \}$, which doesn't contain the curve $C$, could not be convex, a contradiction.

\end{proof}

The sets $V_{\mathcal T}(z_0)$ (orange), $V_{\mathcal{R}^\geq}(z_0)$ (red) and $V_{\mathcal{R}}(z_0)$ (green) are shown in Figure \ref{fig5} for $z_0=\frac1{3}+\frac{i}{2}.$

\begin{figure}[h]
\rule{0pt}{0pt}
\hspace{1.2cm}
\begin{minipage}[hbt]{6cm}
\centering
\includegraphics[width=6cm]{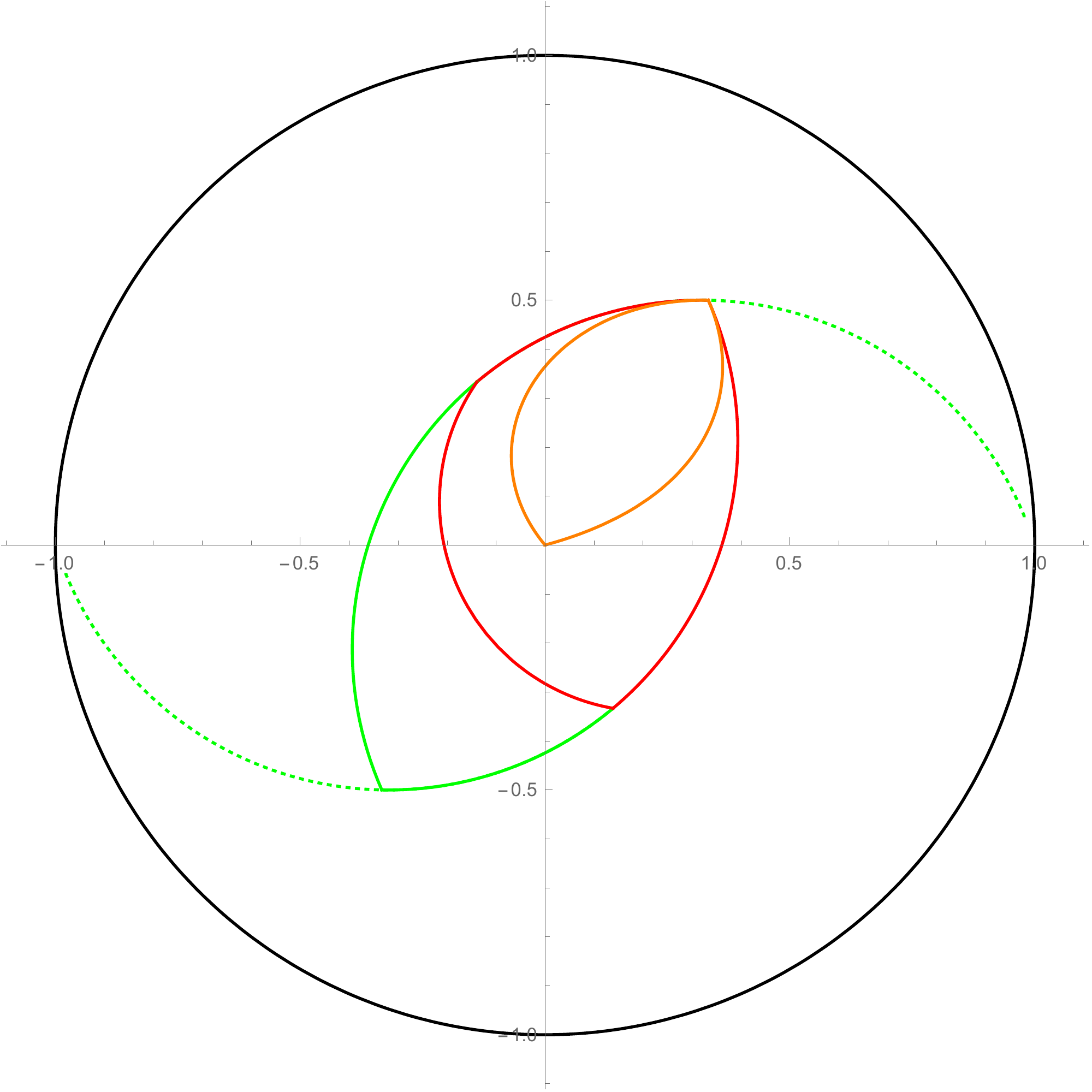}
\caption{$V_{\mathcal T}(z_0), V_{\mathcal{R}^\geq}(z_0), V_{\mathcal{R}}(z_0)$.}
\label{fig5}
\end{minipage}
\hspace{1.2cm}
\begin{minipage}[hbt]{6cm}
\centering
\includegraphics[width=6cm]{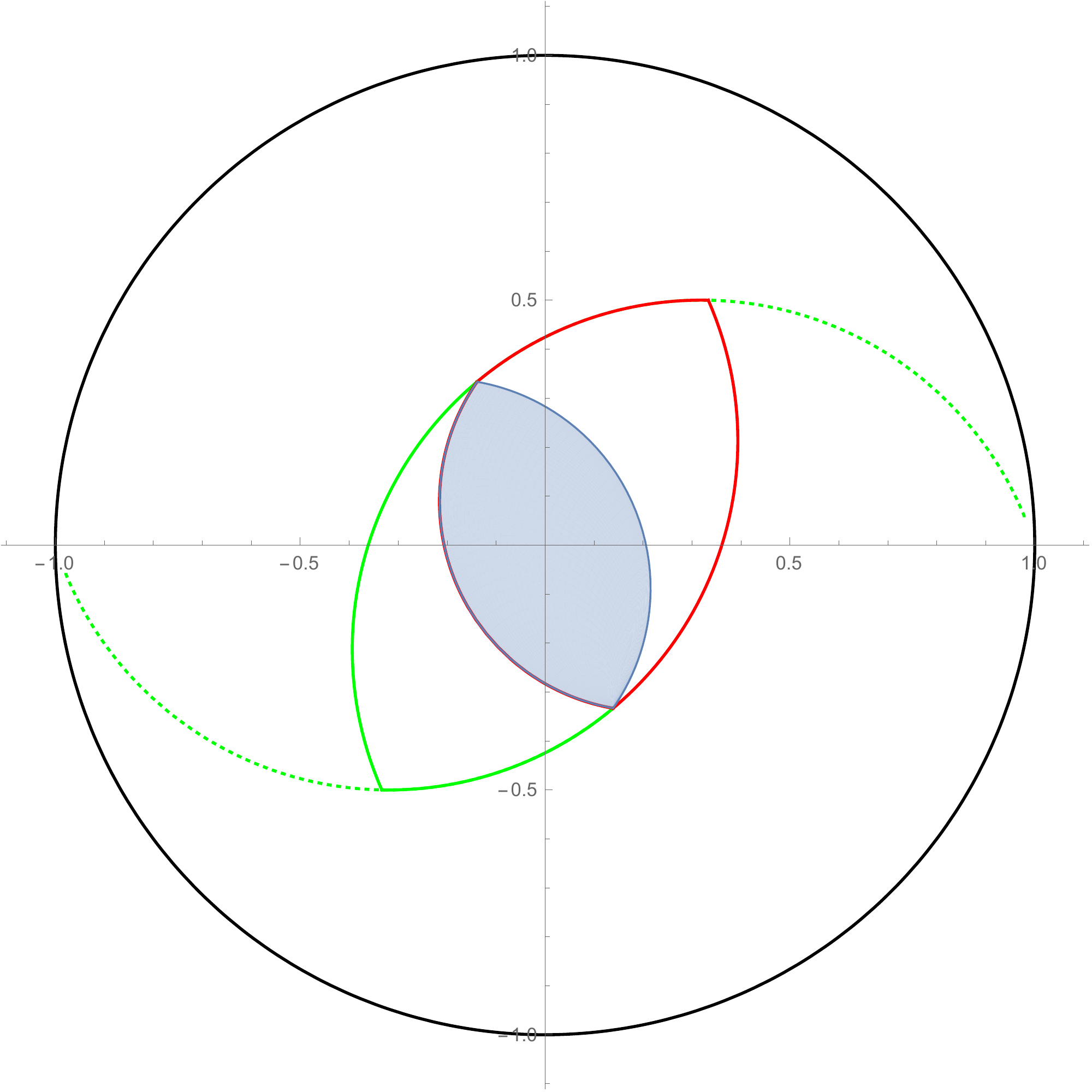}
\caption{$\Delta_0.$}
\label{fig6}
\end{minipage}
\end{figure}

\newpage 

\begin{corollary} Let $\mathcal R^0 = \{f\in \mathcal R^\geq \,|\, f'(0)=0\}$ and $z_0\in\D\setminus\{0\}.$ Then $V_{\mathcal{R}^0}(z_0)=\Delta_0,$ i.e. $V_{\mathcal{R}^0}(z_0)$ is the closed convex set bounded by the circular arcs $C$ and $-C$, which intersect at $z_0^2$ and $-z_0^2.$
\end{corollary}
\begin{proof} We can proceed as in the proof of Theorem \ref{Dost}. Condition \eqref{workhard} then has to be replaced by 
\begin{equation}\label{workhard2}\int_{[0,\pi]} \cos(u) \,\mu(du) =  0, \text{ which we can also write as}\end{equation}
 \begin{equation*}\int_{[0,\pi]} \cos(u) \,\mu(du) \geq  0 \text{ and } \int_{[0,\pi]} \cos(u) \,\mu(du) \leq  0
\end{equation*}
in order to apply again \cite{Winkler}, Theorem 2.1. Then we obtain that the image of $V_{\mathcal{R}^0}(z_0)$ under the map $w\mapsto (1+w)/(1-w)$ is equal to the closure of the convex hull of the set $\{G_\mu(z_0) \,|\, \mu \in S, \mu \text{ satisfies \eqref{workhard2}}\}.$

The proof of Theorem \ref{Dost} shows that this set is equal to $\Delta_0.$
\end{proof}

\begin{corollary} Let $\mathcal R^> = \{f\in \mathcal R^\geq \,|\, f'(0)>0\}$ and $z_0\in\D\setminus\{0\}.$ Then $V_{\mathcal{R}^>}(z_0)=V_{\mathcal{R}^\geq}(z_0)\setminus C,$ where $C$ is the curve from Theorem \ref{Dost}.
\end{corollary}
\begin{proof} Obviously, the curves $A$ and $B$ from Theorem \ref{Dost} minus their endpoint $z_0^2$ and $-z_0^2$ belong to $V_{\mathcal{R}^>}(z_0).$ The curve $C$ does not belong to the set $V_{\mathcal{R}^>}(z_0)$, but it belongs to its closure, which can be seen by approximating the curve by points from \eqref{Set} for $\lambda = \frac{y-1/n}{y-x},$ $n\in\N,$ which means the integral in \eqref{workhard} is equal to $1/n$ in this case. \\
As $V_{\mathcal{R}^>}(z_0)$ is a convex set, we conclude that $V_{\mathcal{R}^>}(z_0)$ is equal to $V_{\mathcal{R}^\geq}(z_0)\setminus C.$

\end{proof}

\end{document}